\documentclass[11pt]{article}
\usepackage{amssymb}
\usepackage{amsmath,amsthm}
\usepackage{verbatim}

\newtheorem{theorem}{Theorem}

\newcounter{Ltheorem}
\newtheorem{lemma}[Ltheorem]{Lemma}
\theoremstyle{definition}\newtheorem{definition}{Definition}

\title{Antiderivatives Exist without Integration}
\author{Charles Coppin}
\begin{document}
\maketitle

\begin{abstract}  
We present a proof that any continuous function with domain including a closed interval yields an antiderivative of that function on that interval.  This is done without the need of any integration comparable to that of Riemann, Cauchy, or Darboux.  The proof is based on one given by Lebesgue in 1905.
\end{abstract}

H. Lebesgue~\cite{Leb} proved in 1905 that an antiderivative of a continuous function with domain including a closed interval exists without knowledge of the Riemann Integral or any of its equivalent forms.  Quoting Lebesgue, 
\begin{quote}Au commencement du cours de calcul int\'{e}gral on d\'{e}montre l'existence de fontions primitives pour les fonctions continues et l'on \'{e}tablit les relations qui lient ces fonctions primitives aux int\'{e}grales d\'{e}finies. La m\'{e}thode universellement adopt\'{e}e pour cela est celle de Cauchy; l'un des avantages de cette m\'{e}thode est de pr\'{e}pararer les g\'{e}n\'{e}ralisations de l'int\'{e}grale qu'ont donn\'{e}es Riemann et M. Darboux.  Cependant, si l'on se limite dans tout le cours aux functions continues, on peut peut-\^{e}tre la remplacer par la suivante \`{a} peine diff\'{e}rente, mais qui me parait plus simple.
\end{quote}

Lebesgue's 1905 paper is in French, which is a primary reason that this beautiful proof of his is not known in the English speaking world.  
Our purpose is two fold: 
\begin{itemize}
\item Give the main construction that Lebesgue used.
\item Outline of a proof using modern ideas normally found in an undergraduate course in analysis.
\end{itemize}
\section*{Lebesgue's Construction}  
 
Throughout, let $f$ be a function that is continuous on an interval $[a, b]$.

We use Lebesgue's notation and give a modern rendition of his construction. Suppose a set of $n$ points $\{(a_{0},d_{0}), (a_{1},d_{1}), \dots, (a_{n},d_{n})\}$,  $a=a_{0} < a_{1} < \dots < a_{n}=b$ are points of the interval $[a,b]$.   If desired, allow $f(a_{i}) = d_{i}, i = 0, 1, \dots, n$.
Lebesgue defines a continuous function $\phi$ with domain including $[a,b]$ such that 
for $i=0,1,\dots,n-1$, there are numbers $m_{i},b_{i}$ such that $\phi(x) = m_{i}x+b_{i}$ for each $x$ in $[a_{i},a_{i+1}]$, and  $m_{i}a_{i+1}+b_{i}=m_{i+1}a_{i+1}+b_{i+1}$ holds for $i=0,1,\dots,n-2$.    
 
Lebesgue defines an antiderivative $\Phi$ for $\phi$ on $[a,b]$ as follows:
\begin{enumerate}
\item Let $\Phi_{0}(x) = (m_{0}/2)x^{2} + b_{0}x - (m_{0}/2)a_{0}^{2}-b_{0}a_{0}$ for each $x$ in $[a_{0},a_{1}]$.  
\item Define $\Phi_{1}(x) = (m_{1}/2)x^{2} + b_{1}x + \Phi_{0}(a_{1}) - (m_{1}/2)a_{1}^{2}-b_{1}a_{1}$ for each $x$ in $[a_{1},a_{2}]$.  Inductively, define $\Phi_{i}, i=2,3,\dots,n-1$.  
\item Since $\Phi_{0}(a_{0}), \Phi_{1}(a_{1}), \dots, \Phi_{n-1}(a_{n-1})$ are well defined, the function $\Phi$ is defined as follows:
\begin{equation*}
\Phi(x) = \begin{cases}
	\displaystyle \frac{m_{0}}{2}x^{2} + b_{0}x - \frac{m_{0}}{2}a_{0}^{2} - b_{0}a_{0}, & \text{if $x \in [a_{0},a_{1}]$}; \\
	& \\
	\displaystyle \frac{m_{i}}{2}x^{2} + b_{i}x +\Phi_{i-1}(a_{i})- \frac{m_{i}}{2}a_{i}^{2} - b_{i}a_{i}, & \text{if $x \in [a_{i},a_{i+1}]$},i=1,2,\dots,n-1.
		\end{cases}
\end{equation*}
\end{enumerate}

The constructed function $\Phi$ consists of $n$ second degree polynomials whose left and right slopes at $a_{1},a_{2},\dots,a_{n-1}$, respectively, are equal. 

From the above construction, it can be shown that there is a function $F$ whose derivative is $f$.  This is the unique contribution of Lebesgue.  The reader may see how the proof might proceed from this point; however, we give an outline below.

\section*{Outline of Proof}

For the following, we use partition of $[a,b]$ to mean any finite collection of subintervals of $[a,b]$ that are non-overlapping and whose union is $[a,b]$.  A refinement $P^{\prime}$ of a partition $P$ is merely another partition of $[a,b]$ such that each end point of each member of $P$ is also an end point of a member of $P^{\prime}$.  For $n=1,2,\dots$, we let $P_{n}$ denote a regular partition of $[a,b]$ with $2^{n-1}$ members each having length $(b-a)/2^{n-1}$. This makes $P_{m}$ a refinement of $P_{n}$ for  positive integers $m,n$ where $m \ge n$.  The functions $\phi_{n}$ and $\Phi_{n}$ will be based on the partition $P_{n}$ for $n =1,2, \dots$.  
When $P_{n}$, $\phi_{n}$, and $\Phi_{n}$ are used, assume that the subscript is a positive integer unless otherwise stated. 

We remind the reader that any continuous function achieves extrema on any closed interval in its domain.  This makes the following definition non-vacuous.

\begin{definition} Suppose $P$ is a partition of $[a,b]$. By the {\bf oscillation} of $f$ on $\delta \in P$ (written $\omega_{\delta}$), we mean the real number
\[\omega_{\delta} =  \underset{\delta}{\max} f - \underset{\delta}{\min} f.\]
Moreover, by the {\bf total oscillation} of $f$ on a partition $P$ of $[a,b]$ (written $\Omega(P)$), we mean  the maximum value of the finite set of oscillations $\{\omega_{\delta} : \delta \in P\}$. 
\end{definition}

Whenever each of $m$ and $n$ is a positive integer with $m \ge n$ and $\Omega_{n} < \epsilon$ for some $\epsilon > 0$, then $\Omega_{m} < \epsilon$. (Remember that $P_{m}$ is a refinement of $P_{n}$.) This result is an application of the previous definitions of oscillation and total oscillation.  We state the following  lemmas without proofs each ultimately being an application of the Heine-Borel Theorem or the basic definitions of oscillation and total oscillation given above. 

\begin{lemma} \label{ep2}
For each $\epsilon >0$, there is a positive integer $n$ such that  
$\Omega_{m} < \epsilon$ for each positive integer $m \ge n$; that is, $\Omega_{n} \rightarrow 0$ as $n \rightarrow \infty$.
\end{lemma}

\begin{lemma} \label{Omega4}  
For each positive integer $n$, $ |f(x) - \phi_{n}(x)| \le \Omega_{n}$ for each $x$ in $[a,b]$; that is, $\phi_{n}(x)$ converges to $f(x)$ for each $x \in [a,b]$.
\end{lemma}

\begin{theorem} The sequence $\{\Phi_{n}\}$ converges uniformly to a function $F$ that is an antiderivative of $f$.
\end{theorem}
\begin{proof}
By Lemma~\ref{Omega4} and Theorem 7.9 of \cite{Rudin}, we know that $\phi_{n} \rightarrow f$ uniformly on $[a,b]$.  By Theorem 7.17 of \cite{Rudin}, $\Phi_{n}$ converges uniformly to $F$ and
\[\phi_{n}(x) \rightarrow F^{\prime}(x)\]
for each $x \in [a,b]$.  Thus, $F^{\prime} = f$ on $[a,b]$.
\end{proof}

Lebesgue finished his paper by proving that the integral of $f$ exists on $[a,b]$. He did this by applying a construction that he developed in his proof that $F^{\prime} = f$.

 \end{document}